\documentclass[11pt, a4paper,leqno]{article}

\usepackage{amsmath, amsthm, amsfonts, amssymb, latexsym, amscd, enumitem, stackrel,  url,hyperref}
\usepackage{euler}
\usepackage{times}

\usepackage{xy}
\xyoption{all}

\usepackage{enumitem}
\setenumerate{label={\rm{(\arabic*)}}}
\setitemize{label=$\circ$}

\makeatletter
\def\url@smallstyle{%
  \@ifundefined{selectfont}{\def\UrlFont{\sf}}{\def\UrlFont{\small\ttfamily}}}
\makeatother
\numberwithin{equation}{section}

\theoremstyle{plain}
    \newtheorem{theorem}[equation]{Theorem}
    \newtheorem{lemma}[equation]{Lemma}
    \newtheorem{corollary}[equation]{Corollary}
    \newtheorem{proposition}[equation]{Proposition}

    \newtheorem*{theorem*}{Theorem}
    \newtheorem*{proposition*}{Proposition}
    \newtheorem*{corollary*}{Corollary}
    \newtheorem*{lemma*}{Lemma}
    \newtheorem*{conjecture*}{Conjecture}
    \newtheorem{definition-theorem}[equation]{Definition/Theorem}
    \newtheorem{definition-lemma}[equation]{Definition/Lemma}
\theoremstyle{definition}
    \newtheorem{definition}[equation]{Definition}
    \newtheorem{example}[equation]{Example}
    \newtheorem{examples}[equation]{Examples}

    \newtheorem{remark}[equation]{Remark}

    \newcommand{\R}{\mathbb{R}}
    \newcommand{\C}{\mathbb{C}}
    \newcommand{\N}{\mathbb{N}}
    \newcommand{\Z}{\mathbb{Z}}
    \newcommand{\Q}{\mathbb{Q}}

   	\renewcommand{\phi}{\varphi}
    \let\epsilon\varepsilon

    \newcommand{\Bounded}{\mathcal{B}}
    \newcommand{\Compact}{\mathcal{K}}
    
 \newcommand{\into}{\hookrightarrow}

\newcommand{\dual}{\vee}

\DeclareMathOperator{\Hom}{Hom}

\DeclareMathOperator{\coker}{coker}

\DeclareMathOperator{\Index}{Index}

\DeclareMathOperator{\Ext}{Ext}

\DeclareMathOperator{\image}{image}

\newcommand{\nonsing}{\operatorname{ns}}
\newcommand{\sing}{\operatorname{s}}
\renewcommand{\mod}{\operatorname{mod}}

\newcommand{\coloneq}{:=}

\newcommand{\ket}[1]{\left|{#1}\right\rangle}

\newcommand{\Rep}{R}

\begin{document}
\title{Fredholm modules over graph $C^*$-algebras}
\author{Tyrone Crisp}
\date{\today}
\maketitle

\begin{abstract}
We present two applications of explicit formulas, due to Cuntz and Krieger, for computations in $K$-homology of graph $C^*$-algebras. We prove that every $K$-homology class for such an algebra is represented by a Fredholm module having finite-rank commutators; and we exhibit generating Fredholm modules for the $K$-homology of quantum lens spaces.
\end{abstract} 

\section*{Introduction}
A graph $C^*$-algebra is a $C^*$-algebra with a presentation determined by a directed graph. There is a close relationship between the combinatorial properties of a graph and the properties of the associated $C^*$-algebra, and this makes graph algebras especially amenable to explicit computations. In this note we study Fredholm modules and $K$-homology for graph $C^*$-algebras, with an emphasis on computations. 

Cuntz and Krieger showed in \cite{CK} and \cite{CK2} that the $\Ext$ groups of a Cuntz-Krieger algebra are isomorphic to the kernel and cokernel of a certain integer matrix. A formula for the isomorphism in odd degree was given in \cite{CK}; the corresponding formula in even degree can easily be recovered from the results of \cite{CK2}. In \cite{Tomforde_Ext}, Cuntz and Krieger's argument was adapted to the $C^*$-algebras of graphs satisfying Condition (L) (every loop has an exit), in which every vertex emits a finite, nonzero number of edges. The latter requirement was removed in \cite{DT}. The assumption of Condition (L) was avoided in \cite{Yi}, using an argument quite different to that of Cuntz and Krieger.

In Section \ref{K_section} we present the formulas deriving from \cite{CK}, \cite{CK2} and \cite{Tomforde_Ext}. The only novelty of our presentation is the translation from the language of extensions to that of Fredholm modules, and the removal of the assumption of Condition (L). We restrict our attention to graphs in which each vertex emits only finitely many vertices, for which the formulas are simplest; arbitrary graphs may be dealt with by similar methods.

We give two applications of these formulas. Firstly, we show that every class in the $K$-homology of a graph $C^*$-algebra $C^*(G)$ can be represented by a Fredholm module $(\rho,H,F)$ in which the commutator $[F,\rho(x)]$ is of trace class for each of the canonical generators $x\in C^*(G)$. This generalises a result of Goffeng and Mesland \cite{GM}, who proved the existence of such modules for all $K^1$ classes, and some $K^0$ classes, over Cuntz-Krieger algebras. Our methods seem to be quite different to those of \cite{GM}.

The second application is a computation of $K$-homology for `quantum' deformations of lens spaces. Hong and Szymanski have shown that the $C^*$-algebras corresponding to these spaces are isomorphic to the $C^*$-algebras of certain finite directed graphs (not satisfying Condition (L)). This turns the computation of the $K$-invariants of these algebras into a matter of finite-dimensional linear algebra: see \cite{HS} and \cite{HSlens}, where some $K$-theory groups are computed, and also \cite{ABL} where the same algebraic computation arises from geometric considerations. In Section \ref{lens_section} we exhibit generators for the $K$-homology of the quantum lens spaces as explicit Fredholm modules. The quantum lens spaces are defined as quotients of odd-dimensional quantum spheres $S_q$ by the action of a finite cyclic group $C$. Hawkins and Landi have constructed Fredholm modules generating the $K$-homology of $S_q$ in \cite{HL}. We prove that the pushforward from $S_q$ to $S_q/C$ is an isomorphism in even $K$-homology, while in odd degree the pushforward of Hawkins and Landi's module decomposes into eigenspaces for the group action, and these eigenspaces generate $K^1(C(S_q/C))$.

\subsubsection*{Acknowledgements}
Thanks to Wojciech Szymanski, who suggested this topic to me (in 2004), and to Francesca Arici, Simon Brain and Giovanni Landi, whose invitation to the workshop on ``$K$-homology and graph algebras'' in Trieste prompted the resurrection of this project. This work was partly supported by an Australian Postgraduate Award, and partly by the Danish National Research Foundation through the Centre for Symmetry and Deformation (DNRF92).

\subsubsection*{Notation and terminology}
A directed graph $G$ consists of (countable) sets $V$ and $E$ of vertices and edges, and maps $s,r:E\to V$ describing the source and range of each directed edge. We assume that $G$ is \emph{row-finite}, meaning that the set $\{e\in E\ |\ s(e)=v\}$ is finite for each $v\in V$. A \emph{sink} is a vertex for which the above set is empty; the set of all sinks is denoted $V_{\sing}$, and its complement in $V$ is denoted $V_{\nonsing}$. 
The $C^*$-algebra $C^*(G)$ associated to $G$ is the universal $C^*$-algebra generated by $V$ and $E$ subject to the relations
\[
 v=v^*=v^2\qquad vw=0 \qquad e^*e=r(e) \qquad e^*f=0 \qquad u=\sum_{s(e)=u} ee^*
\]
for all $v\neq w\in V$, $e\neq f\in E$ and $u\in V_{\nonsing}$.
See \cite{Raeburn} for more on graph $C^*$-algebras.

A \emph{Fredholm module} over a $C^*$-algebra $A$ is a triple $(\rho,H,F)$, where $\rho:A\to \Bounded(H)$ is a $*$-representation of $A$, and $F\in \Bounded(H)$ satisfies $F=F^*$, $F^2=1$, and $[F,\rho(a)]\sim 0$ for each $a\in A$; here $\sim$ denotes equality modulo compact operators. A \emph{graded} Fredholm module is one in which $\rho$ is a $\Z/2$-graded representation, and the operator $F$ reverses the grading. The odd (even) $K$-homology group $K^1(A)$ ($K^0(A)$) is a group of equivalence classes of (graded) Fredholm modules. The equivalence relation can be defined in several equivalent ways: see \cite{HR} and \cite{Blackadar} for details.

\section{Computing $K$-homology}\label{K_section}

Throughout this section, $G=(V,E,s,r)$ is a row-finite directed graph. Let $\Z V$ denote the free abelian group generated by $V$, and let $\Z V_{\nonsing}$ be the subgroup generated by $V_{\nonsing}$.
Consider the map
\[  \Z V_{\nonsing}\xrightarrow{\partial} \Z V\qquad\qquad \partial(v) = \left(\sum_{s(e)=v} r(e)\right) - v\]
which we view as a chain complex $A_*(G)$ concentrated in degrees $1$ and $0$.

\begin{theorem}\label{Ktheory_theorem}\textup{\cite{CK2}, \cite{PR}, \cite{RS}, \cite{DT}, \cite{Yi}.}
 The chain complex $A_*(G)$ computes the $K$-theory of $C^*(G)$:
 \[ K_0(C^*(G))\cong \coker \partial \quad \text{and}\quad K_1(C^*(G))\cong \ker \partial.\]
\end{theorem}

The isomorphism $\beta:\coker\partial\to K_0(C^*(G))$ is given by the following simple formula: 
\begin{equation}\label{beta_equation} \beta(v+\image\partial)=[v]\end{equation}
where $[v]$ denotes the $K_0$-class of the projection $v\in C^*(G)$.

Taking duals (i.e., applying the functor $X^\dual = \Hom_{\Z}(X,\Z)$) gives a cochain complex $A^*(G)$, concentrated in degrees $0$ and $1$:
\[  \Z V^\dual \xrightarrow{\partial^\dual} \Z V_{\nonsing}^\dual\qquad\qquad \partial^\dual \eta (v) = \left(\sum_{s(e)=v} \eta(r(e))\right) - \eta(v).\]

\begin{theorem}\textup{\cite{CK}, \cite{CK2}, \cite{Tomforde_Ext}, \cite{DT}, \cite{Yi}.}
 The cochain complex $A^*(G)$ computes the $K$-homology of $C^*(G)$:
 \[ K^0(C^*(G)) \cong \ker \partial^\dual\quad \text{and}\quad K^1(C^*(G))\cong \coker \partial^\dual.\]
\end{theorem}

One can extract explicit formulas for these isomorphisms from the arguments of Cuntz and Krieger (and their later generalisations to graph $C^*$-algebras). In even degree one has:

\begin{theorem}\label{K0_theorem}\textup{\cite{CK2}.}
 Let $\mathcal F=\left( \rho_0 \oplus \rho_1, H_0\oplus H_1, F \right)$ be a graded Fredholm module over $C^*(G)$. The isomorphism $K^0(C^*(G))\to \ker \partial^\dual$ sends the class of $\mathcal F$ to the function
 \[ \Index_{\mathcal F}(v)= \Index \left(\rho_1(v) F \rho_0(v)\right),\]
 the Fredholm index of $\rho_1(v) F \rho_0(v)$ as an operator from $\rho_0(v)H_1$ to $\rho_1(v)H_0$.
\end{theorem}

Theorem \ref{K0_theorem} is proved in Section \ref{K0_index_section}; it is an easy consequence of \eqref{beta_equation} and the UCT (\cite[Theorem 3.11]{CK2}, \cite{RosScho}). The theorem is used in Section \ref{K0_module_section} to construct an explicit representative for each class in $K^0(C^*(G))$.

The formula for the isomorphism $K^1(C^*(G))\cong \coker \partial^\dual$ relies on the following lemma:

\begin{lemma}\label{proj_degenerate_lemma}
Let $\mathcal F$ be a Fredholm module over $C^*(G)$, where $G$ is a row-finite directed graph. There is a compact perturbation $(\rho,H,F)$ of $\mathcal F$ satisfying
\begin{equation}\label{proj_degenerate_equation}\tag{$\star$} [ F,\rho(ee^*) ] = [F, \rho(v) ] = 0 \quad \text{for all}\quad e\in E,\ v\in V.\end{equation}
\end{lemma}

(Recall that $(\rho,H,F)$ is a \emph{compact perturbation} of $(\rho,H,F')$ if $(F-F')\rho(a)\sim 0$ for every $a\in A$.)

\begin{proof}
Let $\mathcal F=(\rho,H,F')$, and let $P=\frac{1}{2}(F'+1)$. For each $e\in E$, the operator $\rho(ee^*)P\rho(ee^*)$ descends to a projection in the Calkin algebra of $\rho(ee^*)H$, and a standard functional-calculus argument shows that there is a projection $q_{e}$ on $\rho(ee^*)H$ having $q_e\sim \rho(ee^*)P\rho(ee^*)$. Similarly, for each sink $v\in V_{\sing}$ there is a projection $q_v$ on $\rho(v)H$ having $q_v\sim \rho(v)P\rho(v)$. Let $Q= \sum_{e\in E} q_e + \sum_{v\in V_{\sing}} q_v$, and take $F=2Q-1$.
\end{proof}

\begin{theorem}\label{K1_theorem} \textup{\cite{CK}, \cite{Tomforde_Ext}.}
Let $G$ be a row-finite directed graph, and let $\mathcal F=(\rho,H,F)$ be a Fredholm module over $C^*(G)$ satisfying \eqref{proj_degenerate_equation}. The isomorphism $K^1(C^*(G))\to \coker \partial^\dual$ sends the class of $\mathcal F$ to the class of the function
  \[
   \Index_{\mathcal F} (v) = \sum_{s(e)=v} \Index (P\rho(e)P), 
  \]
where  $P=\frac{1}{2}(F+1)$, and $\Index(P\rho(e)P)$ is the Fredholm index of $P\rho(e)P$ as an operator from $\rho(e^*e)PH$ to $\rho(ee^*)PH$. 
 \end{theorem}

 Theorem \ref{K1_theorem} is proved in Sections \ref{K1_index_section} and \ref{K1_module_section}. 
 
 The constructions in Sections \ref{K0_module_section} and \ref{K1_module_section} yield the following corollary:
 
 \begin{corollary}\label{summable_corollary}
 Let $G$ be a row-finite directed graph. Every class in $K^*(C^*(G))$ can be represented by a Fredholm module $(\rho,H,F)$ for which the commutators $[F,\rho(x)]$ are finite-rank operators for all $x\in V\sqcup E$.
 \end{corollary}

 The same result has been proved for $K^1$ of Cuntz-Krieger algebras (and for $K^0$ in special cases) in \cite{GM}, by quite different methods.

\subsection{The index map for $K^0$}\label{K0_index_section}

The even $K$-homology group $K^0(C^*(G))$ is computed using the following diagram:
\[
 \xymatrix{
 K^0(C^*(G)) \ar[r]^-{\alpha} \ar[d]_-{\Index} & K_0(C^*(G))^\dual \ar[d]^-{\beta^\dual} \\
 \ker \partial^\dual \ar[r]^-{\gamma} & (\coker \partial)^\dual
 }
\]
Here $\alpha$ is the index pairing between $K$-homology and $K$-theory: if $p\in C^*(G)$ is a projection, and $\mathcal F=\left( \rho_0 \oplus \rho_1, H_0\oplus H_1, F \right)$ is a graded Fredholm module over $C^*(G)$, then 
\begin{equation}\label{alpha_equation} \alpha_{\mathcal F}(p) = \Index \left(\rho_0(p)H_0\xrightarrow{\rho_1(p) F \rho_0(p)} \rho_1(p)H_1\right).\end{equation}
Since $K_1(C^*(G))$ is torsion-free, the UCT implies that $\alpha$ is an isomorphism (see \cite{HR}, \cite{RosScho}). The map $\gamma$ is the isomorphism induced by the canonical pairing $\Z V^\dual \times \Z V \to \Z$. The map $\beta^\dual$ is the dual of the isomorphism $\beta:\coker \partial \to K_0(C^*(G))$ from \eqref{beta_equation}. The isomorphism $\Index$ is defined by insisting that the diagram commute. The formula appearing in Theorem \ref{K0_theorem} follows immediately from \eqref{beta_equation} and \eqref{alpha_equation}.

\subsection{Explicit Fredholm modules for $K^0$}\label{K0_module_section}

Let $G=(V,E,s,r)$ be a row-finite directed graph, and let $\eta$ be an element of $\ker \partial^\dual$: that is, a function $V\to \Z$ such that
\begin{equation}\label{harmonic_equation}\eta(v)=\sum_{s(e)=v} \eta(r(e))\quad\text{for all }v\in V_{\nonsing}.\end{equation} 
We shall construct a graded Fredholm module $\mathcal F$ having $\Index_{\mathcal F}=\eta$.

Let $H=\ell^2(\Z)\otimes \ell^2(V)$ be a Hilbert space with orthonormal basis $\{ \ket{n,v}\ |\ n\in \Z, v\in V\}$. For each vertex $v\in V$, let $\rho_0(v)\in \Bounded(H)$ be the projection of $H$ onto its subspace $\ell^2(\{n\geq 0\})\otimes \ell^2(\{v\})$, and let $\rho_1$ be the projection onto the subspace $\ell^2(\{n\geq \eta(v)\})\otimes \ell^2(\{v\})$. 
For each vertex $v$ that is not a sink, choose an ordering $e_0,\ldots,e_{d-1}$ of the edges with source $v$. For such an edge $e_i$, let $\rho_0(e_i)\in \Bounded(H)$ be the partial isometry with support $\rho_0(r(e_i))$, given by $\rho_0(e_i)\ket{n, r(e_i)} = \ket{i+nd,v}$. For each $i$, choose a function $b_i:\{n\geq \eta(r(e_i))\} \to \{n\geq \eta(v)\}$ satisfying $b_i(n)=i+nd$ for all $n\geq 0$, and such that the disjoint union
\[ \sqcup b_i : \bigsqcup_{i=1}^{d-1}\{n\geq \eta(r(e_i))\} \to \{n\geq \eta(v)\} \]
is a bijection: this is possible by virtue of \eqref{harmonic_equation}. Then let $\rho_1(e_i)$ be the partial isometry with support $\rho_1(r(e_i))$, given by $\rho_1(e_i)\ket{n,r(e_i)} = \ket{b_i(n),v}$. 

The maps $\rho_0,\rho_1:V\sqcup E\to \Bounded(H)$ extend to $*$-representations of $C^*(G)$. Each $\rho_1(v)$ is a finite-rank perturbation of $\rho_0(v)$, and each $\rho_1(e)$ is a finite-rank perturbation of $\rho_0(e)$, so 
\[ \mathcal F = \left( \rho_0\oplus \rho_1, H\oplus H, \begin{bmatrix} 0 & 1\\ 1& 0\end{bmatrix}\right)\]
is a graded Fredholm module over $C^*(G)$. For each vertex $v$, the Fredholm operator $\rho_1(v)\rho_0(v):\rho_0(v)H\to \rho_1(v)H$
has index $\eta(v)$, and so Theorem \ref{K0_theorem} implies that
$\Index_{\mathcal F} = \eta$. This proves the even case of Corollary \ref{summable_corollary}.

\subsection{From vertices to edges}

The complex $A_*(G)$, which is defined in terms of the vertices of $G$, can be replaced by a complex defined in terms of the edges. The latter complex will be used in the proof of Theorem \ref{K1_theorem}. The complex $A_*(G)$ is the more useful for practical computations, since most graphs have more edges than vertices.

Consider the map
\[
 \Z E \xrightarrow{d} \Z[E\sqcup V_{\sing}] \qquad d(e)=\begin{cases} \left(\sum_{s(f)=r(e)} f\right) -e &\text{if }r(e)\in V_{\nonsing}\\ r(e)-e&\text{if }r(e)\in V_{\sing}\end{cases}
\]
which we view as a chain complex $B_*(G)$ concentrated in degrees $1$ and $0$. Define maps $\sigma:A_*(G)\to B_*(G)$ and $\tau:B_*(G)\to A_*(G)$ as follows:
\[
 \sigma_*(v)=\begin{cases} \sum_{s(e)=v} e &\text{if }v\in V_{\nonsing}\\ v&\text{if }v\in V_{\sing}\end{cases} \qquad
 \tau_1(e)=\begin{cases} r(e)&\text{if }r(e)\in V_{\nonsing}\\ 0&\text{if }r(e)\in V_{\sing}\end{cases}\]
 \[\tau_0(e)=r(e)\quad \text{for }e\in E,\qquad \tau_0(v)=v\quad \text{for }v\in V_{\sing}.\]
 
 \begin{lemma}\label{qis_lemma}
  $\sigma$ and $\tau$ are mutually homotopy-inverse quasi-isomorphisms.
 \end{lemma}

 \begin{proof}
 That $\sigma$ and $\tau$ are maps of chain complexes is easily verified. Let $h:\Z[E\sqcup V_{\sing}]\to \Z E$ be the map $e\mapsto e$, $v\mapsto 0$. Let $k:\Z V \to \Z V_{\nonsing}$ be the map $v\mapsto v$ if $v\in V_{\nonsing}$, $v\mapsto 0$ if $v\in V_{\sing}$. Simple computations show that $h$ is a homotopy from $\sigma\tau$ to the identity on $B_*(G)$, and $k$ is a homotopy from $\tau\sigma$ to the identity on $A_*(G)$. 
 \end{proof}

 \subsection{The index map for $K^1$}\label{K1_index_section}

 This section is essentially a translation of Cuntz and Krieger's argument \cite[Section 5]{CK} (as adapted to graph algebras by Tomforde in \cite{Tomforde_Ext}) from the language of extensions to that of Fredholm modules; the main novelty here is the avoidance of condition (L), which is achieved by the following Lemma. Recall that a graph satisfies Condition (L) if it contains no \emph{loop without exit}; a loop without exit is a sequence of edges $e_1e_2\cdots e_n$ such that $s(e_{i+1})=r(e_i)$, $s(e_1)=r(e_n)$, and for each $i$ the vertex $s(e_i)$ is not the source of any edge besides $e_i$. Let us say that a $*$-representation $\phi:C^*(G)\to \Bounded(H)$ is \emph{ample} if $\phi(a)\sim 0$ implies $a=0$. 
 
 \begin{lemma}\label{technical_lemma}
 Let $G$ be a row-finite directed graph, $H$ a separable Hilbert space, and $\{p_x\in \Bounded(H)\ |\ x\in E\sqcup V_{\sing}\}$ a family of mutually orthogonal infinite-rank projections. There is an ample representation $\phi:C^*(G)\to \Bounded(H)$ satisfying $\phi(ee^*)=p_e$ and $\phi(v)=p_v$ for all $e\in E$ and $v\in V_{\sing}$. 
 \end{lemma}

 \begin{proof}
 For each $v\in V_{\nonsing}$ let $p_v=\sum_{s(e)=v} p_e$. Then for each edge $e$, choose a unitary isomorphism $\phi(e):p_{r(e)}H \to p_e H$. This may be done in such a way that for each loop $l=e_1 e_2 \cdots e_n$ in $G$ without an exit, the essential spectrum of the operator $\phi(e_1)\phi(e_2)\cdots\phi(e_n)$ contains the unit circle: indeed, one can choose $\phi(e_1),\ldots,\phi(e_{n-1})$ arbitrarily, and then let $\phi(e_n) = \left(\phi(e_1)\cdots \phi(e_{n-1})\right)^{-1} u$, where $u: p_{e_1} H\to p_{e_1}H$ is a unitary with essential spectrum $\mathbb{T}$. The universal property of $C^*(G)$ implies that the map $e\mapsto \phi(e)$, $v\mapsto p_v$ extends to a $*$-representation $\phi:C^*(G)\to \Bounded(H)$, and Szymanski's generalisation of Cuntz and Krieger's uniqueness theorem \cite[Theorem 1.2]{SzCK} ensures that the composition of $\phi$ with the quotient map to the Calkin algebra is injective, so that $\phi$ is ample.
 \end{proof}

 For each Fredholm module $\mathcal F=(\rho, H, F)$ over $C^*(G)$ satisfying \eqref{proj_degenerate_equation}, we consider the function $\Index_{\mathcal F}:E\to \Z$ defined by
 \[ \Index_{\mathcal F}(e) = \Index\left( \rho(e^*e) PH \xrightarrow{P\rho(e)P} \rho(ee^*)PH\right)\]
 where $P$ is the projection $\frac{1}{2}(F+1)$.

 \begin{lemma}\label{K1_welldefined_lemma}
 Let $\mathcal F$ be a Fredholm module over $C^*(G)$ satisfying \eqref{proj_degenerate_equation}, and suppose that $\mathcal E$ is a compact perturbation of $\mathcal F$ also satisfying \eqref{proj_degenerate_equation}. Then $\Index_{\mathcal F}=\Index_{\mathcal E}$ in $\coker d^\dual$. 
 \end{lemma}

 \begin{proof}
Let $P$ and $Q$ be the projections associated to $\mathcal F$ and $\mathcal E$, respectively. By adding to $\mathcal F$ and $\mathcal E$ a degenerate Fredholm module of the form $(\alpha, H_\alpha, 1)$, where $\alpha$ is an infinite direct sum of faithful representations of $C^*(G)$, we may assume that the projections $\rho(ee^*)P$, $\rho(ee^*)Q$, $\rho(v)P$ and $\rho(v)Q$ all have infinite rank. Lemma \ref{technical_lemma} gives ample $*$-representations $\phi:C^*(G)\to \Bounded(PH)$ and $\psi:C^*(G)\to \Bounded(QH)$ satisfying
$\phi(ee^*)=\rho(ee^*)P$, $\phi(v)=\rho(v)P$, $\psi(ee^*)=\rho(ee^*)Q$ and $\psi(v)=\rho(v)Q$ for all $e\in E$ and $v\in V$.
Since $\phi(e^*)$ is an isomorphism from $\rho(ee^*)PH$ to $\rho(e^*e)PH$, we have
\[ \Index_{\mathcal F}(e) = \Index \left( \rho(ee^*)PH \xrightarrow{ P\rho(e)\phi(e^*) P} \rho(ee^*)PH \right).\]
The argument of \cite[Proposition 5.2]{CK} and \cite[Proposition 4.12]{Tomforde_Ext} now applies verbatim: Voiculescu's theorem (see e.g. \cite[Theorem 3.4.6]{HR}) gives a unitary $U:PH \to QH$ satisfying $U\phi(a)U^*\sim \psi(a)$ for all $a\in C^*(G)$, and a short computation shows that $\Index_{\mathcal F} - \Index_{\mathcal E} = d^\dual \eta$,
where $\eta:E\sqcup V_{\sing}\to \Z$ is the function
\begin{align*} \eta(e) &= \Index \left( \rho(ee^*)PH \xrightarrow{ \rho(ee^*) P U P\rho(ee^*)} \rho(ee^*) PH\right)\\
 \eta(v) &=\Index \left(\rho(v)PH \xrightarrow{\rho(v)P U P\rho(v)} \rho(v)PH\right)
 \end{align*}
for $e\in E$ and $v\in V_{\sing}$. Thus $\Index_{\mathcal F}=\Index_{\mathcal E}$ in $\coker d^\dual$.
\end{proof}

\begin{definition} 
 Given an arbitrary Fredholm module $\mathcal F$ over $C^*(G)$, let $\Index_{\mathcal F}\in \coker d^\dual$ be the class of $\Index_{\mathcal E}$, for any compact perturbation $\mathcal E$ of $\mathcal F$ satisfying \eqref{proj_degenerate_equation}.
\end{definition}

 It is clear from the definition that $\Index_{\mathcal F}$ is additive with respect to direct sums of Fredholm modules, and is invariant under unitary equivalences, compact perturbations, and addition of degenerate modules. Therefore the map $\mathcal F\mapsto \Index_{\mathcal F}$ induces a group homomorphism $K^1(C^*(G))\to \coker d^\dual$. The argument of \cite[Proposition 4.20]{Tomforde_Ext} adapts without difficulty to the present setting, showing that the index map is one-to-one. Composition with the quasi-isomorphism $\sigma^\dual:\Z E^\dual \to \Z V_{\nonsing}^\dual$ (Lemma \ref{qis_lemma}) gives an injective group homomorphism
\[ \Index: K^1(C^*(G)) \to \coker\partial^\dual \qquad \Index_{\mathcal F}(v) = \sum_{s(e)=v} \Index_{\mathcal F}(e)\]
which we will show to be surjective in the next section.

\subsection{Explicit Fredholm modules for $K^1$}\label{K1_module_section}

Let $\eta\in \Z V_{\nonsing}^\dual$ be an integer-valued function on the set of nonsingular vertices of $G$. Let $H=\ell^2(\Z)\otimes \ell^2(V)$ be a Hilbert space with orthonormal basis $\{ \ket{n,v}\ |\ n\in \Z,\ v\in V\}$. For each vertex $v\in V$, let $\rho(v)$ denote the orthogonal projection of $H$ onto its subspace $\ell^2(\Z)\otimes \ell^2(\{v\})$.

For each $v\in V_{\nonsing}$, choose an ordering $e_0,\ldots,e_{d-1}$ of the edges with source $v$, and then let $\rho(e_i):\rho(r(e_i))H\to \rho(v)H$ be the isometry
\[ \rho(e_i)\ket{n,r(e_i)} = \begin{cases} \ket{d(n-\eta(v)), v}&\text{if }i=0\\ \ket{i+dn,v}&\text{otherwise.}\end{cases}\]
The map $v\mapsto \rho(v)$, $e\mapsto \rho(e)$ extends to a $*$-representation $\rho:C^*(G)\to \Bounded(H)$.

Let $F \in \Bounded(H)$ be the involution  
\[ F\ket{n,v} =\begin{cases} \phantom{-}\ket{n,v} &\text{if }n\geq 0\\ -\ket{n,v}& \text{if }n< 0.\end{cases}\]
This operator commutes with each $\rho(e_0)$ modulo finite-rank operators, and commutes exactly with each $\rho(v)$, each $\rho(e_i)$ for $i\geq 1$, and each $\rho(ee^*)$. Thus $\mathcal F=(\rho,H,F)$ is a Fredholm module over $C^*(G)$ satisfying \eqref{proj_degenerate_equation}. Let $P=\frac{1}{2}(F+1)$. For each nonsingular vertex $v$, emitting edges $e_0,\ldots,e_{d-1}$, the operator $P\rho(e_0)P$ has index $\eta(v)$, while the operators $P\rho(e_i)P$ for $i\geq 1$ all have index $0$. We conclude that $\Index_{\mathcal F} = \eta$ in $\Z V_{\nonsing}^\dual$. This completes the proof of Theorem \ref{K1_theorem}, and also establishes the odd case of Corollary \ref{summable_corollary}.

\section{Application to quantum lens spaces}\label{lens_section}

\subsection{Background}

In the literature one may find several definitions of quantum lens spaces, and of the quantum spheres of which they are quotients. The following are the definitions that we shall use; they are taken from \cite{VS} and \cite{HSlens}.

\begin{definition}\label{q_definition}
Let $n\geq 2$ be an integer, and $q\in [0,1)$. The $C^*$-algebra $C(S^{2n-1}_q)$ of continuous functions on the quantum $(2n-1)$-sphere is the universal $C^*$-algebra generated by elements $z_1,\ldots,z_n$ with relations
\begin{gather*}
z_j z_i = qz_i z_j\quad (i<j) \quad\qquad z_j^*z_i = qz_i z_j^*\quad (i\neq j) \quad\qquad \sum_{i\geq 1} z_i z_i^* =1\\ z_i^*z_i=z_iz_i^*+(1-q^2)\sum_{j>i} z_jz_j^*.
\end{gather*}
Let $p\geq 1$ be an integer and $\zeta_p\in \C$ a primitive $p$th root of unity. The $C^*$-algebra $C(L^{2n-1}_{q,p})$ of continuous functions on the quantum lens space $L^{2n-1}_{q,p}$ is the fixed-point algebra $C(S_q^{2n-1})^{\alpha_p}$ for the automorphism $\alpha_p(z_i)=\zeta_p z_i$. 
\end{definition}

\begin{examples}
Special cases of the above definitions include quantum analogues of (the topological spaces underlying) the compact Lie groups $\operatorname{SU}(2)$ and $\operatorname{SO}(3)$, and the odd-dimensional real projective spaces $\R\mathrm{P}^{2n-1}$: 
\[ C(S^3_q)\cong C(\operatorname{SU}_q(2))\qquad C(L^3_{q,2})\cong C(\operatorname{SO}_q(3))\qquad C(L^{2n-1}_{q,2})\cong C(\R\mathrm{P}^{2n-1}).\]
See \cite{Wor}, \cite{Lance}, \cite{Podles}, \cite{HS}.
\end{examples}

Setting $q=1$ in Definition \ref{q_definition} gives the (commutative) $C^*$-algebras of continuous functions on the classical odd-dimensional spheres and lens spaces. 
Graph $C^*$-algebras are far from being commutative, which makes the following result of Hong and Szymanski rather surprising.

\begin{definition}
 For $n\geq 2$ let $G_n=(V_n,E_n,s,r)$ be the directed graph
 \[V_{n}=\{v_1,\ldots,v_n\}\qquad  E_{n} = \{e_{i,j}\ |\ 1\leq i \leq j\}\qquad s(e_{ij})=v_i\qquad r(e_{ij})=v_j.\]
 Then for each $p\geq 2$ let $G_{n}^p=(V_n^p, E_n^p,s,r)$ be the graph of length-$p$ directed paths in $G_n$: thus $G_{n}^p$ has vertices $V_n^p=V_n$, and one edge with source $v_i$ and range $v_j$ for each degree-$p$ noncommutative monomial $e_1\cdots e_p$ on $E_n$ having $s(e_1)=v_i$, $r(e_p)=v_j$, and $s(e_{i+1})=r(e_i)$. 
\end{definition}

\begin{theorem}\textup{\cite{HS}, \cite{HSlens}.}
There are isomorphisms of $C^*$-algebras
\[ C(S^{2n-1}_q)\cong C^*(G_n) \qquad \text{and}\qquad C(L^{2n-1}_{q,p})\cong C^*(G_n^p).\] 
\end{theorem}

As a final piece of background to our computation, let us recall the Fredholm modules over quantum spheres constructed by Hawkins and Landi in \cite{HL}. We shall present the formulas in terms of the Cuntz-Krieger generators, as found in \cite{HS}. 

Fix $n\geq 2$, and let $H$ be the Hilbert space $\ell^2(\N^{n-1}\times \Z)$, with orthonormal basis $\{\ket{k}\ |\ k= (k_1,\ldots,k_n) \in \N^{n-1}\times \Z\}$. For each $i=1,\ldots,n$ define operators $\delta_i,\epsilon_i\in \Bounded(H)$ by
\[
 \delta_i \ket{k} =\begin{cases} \ket{k}&\text{if }k_i=0\\ 0&\text{if }k_i\neq 0\end{cases}\qquad \epsilon_i\ket{k} = \ket{k_1,\ldots,k_{i-1},k_i+1,k_{i+1},\ldots,k_n}\] 
and then define a representation $\rho:C^*(G_n)\to \Bounded(H)$ by
\[
\rho(v_n) = \delta_1\delta_2\cdots \delta_{n-1}\qquad \rho(v_i) = \delta_i\cdots \delta_{i-1} (1-\delta_i) \qquad \rho(e_{lj}) =   \epsilon_l\rho(v_j)\]
for $i<n$ and $1\leq l\leq j\leq n$. Let $F\in \Bounded(H)$ be the operator
\[ F\ket{k} = \begin{cases} \phantom{-}\ket{k}&\text{if }k_n\geq 0\\ -\ket{k}&\text{if }k_n < 0.\end{cases}\]
Hawkins and Landi show that $K^1(C(S^{2n-1}_q))\cong \Z$ is generated by the class of the Fredholm module
\begin{equation}\label{HLodd} \mathcal F = (\rho, H, F).\end{equation}
(This fact can also be proved by a computation using  Theorem \ref{K1_theorem}; note that $\mathcal F$ satisfies the condition \eqref{proj_degenerate_equation}.)

For the even $K$-homology, let $\psi:C^*(G)\to \C$ be the character $\psi(v_1)=\psi(e_{11})=1$, $\psi(v_i)=\psi(e_{ij})=0$ for all $i,j\geq 2$. Hawkins and Landi show that $K^0(C(S^{2n-1}_q))\cong \Z$ is generated by the class of
\begin{equation}\label{HLeven} \mathcal E = (\psi\oplus 0, \C\oplus 0, 0).\end{equation}

\subsection{Fredholm modules over quantum lens spaces}

Let us now fix $n\geq 2$, $q\in [0,1)$ and $p\geq 2$, and simply write $S$ for $S^{2n-1}_q$, $L$ for $L^{2n-1}_{q,p}$, $G=(V,E,s,r)$ for $G_n$, and $G^p$ for $G_n^p$. Let $t:\Z V^\dual\to \Z V^\dual$ be the operator 
\[t\eta (v_i) = \begin{cases}\eta(v_{i+1})&\text{if }1\leq i <n\\ 0&\text{if }i=n\end{cases}\] and let $D=1+t+\ldots+ t^{n-1}=(1-t)^{-1}$. Then
\[ D\eta (v) = \sum_{\substack{e\in E\\ s(e)=v}} \eta(r(e))\]
by the definition of the graph $G$, and it follows by induction that
\[ D^p \eta (v) = \sum_{\substack{e\in E^p\\ s(e)=v}} \eta(r(e)).\]
The coboundary in the complex $A^*(G^p)$ is thus given by $\partial_p^\dual = D^p-1$. For each $i=1,\ldots,n$, let $\eta_i\in \Z V^\dual$ be the function which is $1$ on $v_i$ and $0$ on $v_j$ for $j\neq i$. 

The even $K$-homology of the quantum lens spaces is easily computed:

\begin{proposition}\label{lens_K0_proposition}
The class of the Fredholm module $\mathcal E$ of \eqref{HLeven}, restricted from $C(S)$ to $C(L)$, generates $K^0(C(L))\cong \Z$.
\end{proposition}

\begin{proof}
We have
\[ \ker (D^p -1 ) = \ker(1-D^{-p}) = \ker \left(\sum_{i=0}^{p-1} (1-t)^i\right)t = \ker t\]
because the operator $\sum_{i=0}^{p-1} (1-t)^i$ is one-to-one (it has determinant $p^n$). 
The kernel of $t$ is generated by the function $\eta_1=\Index_{\mathcal E}$, and so the proposition follows from Theorem \ref{K0_theorem}.
\end{proof}

Turning to odd degree, let us first observe that the Fredholm module $\mathcal F$ is equivariant:

\begin{lemma}
Let $\mathcal F=(\rho,H,F)$ be as in \eqref{HLodd}, and define a unitary $\alpha_p\in \Bounded(H)$ by 
$\alpha_p\ket{k} = \zeta_p^{k_1+\ldots +k_n}\ket{k}$. Then $[F,\alpha_p]=0$, and $\alpha_p\rho(a)\alpha_p^{-1} = \rho(\alpha_p(a))$ for all $a\in C(S)$.
\qed
\end{lemma}

The restriction of $\mathcal F$ to the subalgebra $C(L) =C(S)^{\alpha_p}$ thus decomposes over the spectrum of $\alpha_p$:
\[\mathcal F = \bigoplus_{m=0}^{p-1} \mathcal F_m \qquad \mathcal F_m \coloneq (\rho, H_m, F)\qquad H_m \coloneq \{h\in H\ |\ \alpha_p h = \zeta_p^m h\}.\]

\begin{proposition}\label{lens_K1_proposition}
The classes of the Fredholm modules $\mathcal F_m$ ($0\leq m\leq p-1$) generate $K^1(C(L))$. Each of these generators has infinite order, and each of the elements $\mathcal F_m-\mathcal F_0$ has finite order. 
\end{proposition}

\begin{proof}
 We will of course use the isomorphism $\Index:K^1(C^*(G^p))\to \coker (D^p-1)$ of Theorem \ref{K1_theorem}. The functions
$D^i\eta_n$, for $i=0,\ldots,n-1$, constitute a basis for $\Z V^\dual$, and so the classes of $D^i\eta_n$ for $i=0,\ldots,p-1$ generate $\coker (D^p-1)$.
The class of $D^i\eta_n$ has infinite order in $\coker (D^p-1)$, because $(D^p-1)\eta(v_n)=0$ for every $\eta$, while $D^i\eta_n(v_n)=1$. The restriction of $D^p-1$ to an operator $\Q[V\setminus\{v_1\}]^\dual \to \Q[V\setminus\{v_n\}]^\dual$ is invertible (it has determinant $p^{n-1}$). Each function $D^i\eta-\eta$ is supported on $\{v_1,\ldots,v_{n-1}\}$, and so the classes $D^i\eta-\eta$ are torsion in $\coker (D^p-1)$. These considerations show that it will suffice to prove that
\begin{equation}\label{IndexF_equation}
 \Index_{\mathcal F_m} = -D^{m}\eta_n\quad \text{for all}\quad m=0,\ldots,p-1.
\end{equation}

Fix $i=1,\ldots,n$ and consider an edge in $G^p$ with source $v_i$: that is, a length-$p$ path $\mu$ in $G$ with source $v_i$. Since the projection $P=\frac{1}{2}(F+1)$ commutes with $\rho(e_{kj})$ unless $k=j=n$, we will have $\Index_{\mathcal F_m}(\mu)=0$ unless $\mu$ is of the form $\lambda e_{nn}^d$ for some $d=1,\ldots,p$, and some path $\lambda$ having source $v_i$, range $v_n$, length $l=p-d$, and not containing the edge $e_{nn}$. (When $i=n$, we allow the ``length-zero path'' $\lambda=v_n$.) Since $P\rho(\lambda)P$ is an isomorphism, we have
\begin{align*} \Index_{\mathcal F_m}(\mu) & = \Index\left(\rho(v_n)PH_m\xrightarrow{P\rho(\mu)P} \rho(\mu\mu^*)PH_m\right)\\
 &= \Index\left( \rho(v_n)PH_m \xrightarrow{P\rho(e_{nn}^d) P} \rho(v_n)PH_{m+d}\right)
\end{align*}
where $m+d$ is taken modulo $p$. The space $\rho(v_n)P H_m$ is the span of the basis vectors $\ket{k}$ having $k_1=\ldots =k_{n-1}=0$, $k_n \geq 0$ and $k_n = m\,\mod p$, and the operator $P\rho(e_{nn}^d)P$ acts by increasing $k_n$ by $d$. This operator is injective; it is surjective if $m+d< p$, otherwise it has one-dimensional cokernel. This shows that $\Index_{\mathcal F_m}$ is equal to $-1$ times the number of paths $\lambda$ as above with length $l\leq m$. Each such $\lambda$ extends uniquely to a length-$m$ path $\lambda e_{nn}^{m-l}$, and every length-$m$ path from $v_i$ to $v_n$ arises in this way. Therefore
\[ \Index_{\mathcal F_m}(v_i) = -\ \#\left\{\text{length-$m$ paths in $G$ from $v_i$ to $v_n$}\right\} = - D^{m}\eta_n(v_i),\]
proving \eqref{IndexF_equation}.
\end{proof}

\begin{example}
For $n=2$ (i.e., for the $3$-dimensional quantum lens space $L^3_{q,p}$) one has $D^p=(1+t)^p=1+pt$. The cokernel of $\partial^\dual=pt$ is generated by $\eta_2$ (which has infinite order) and $\eta_1$ (with has order $p$). The formula \eqref{IndexF_equation} gives $\Index_{\mathcal F_m} = -m\eta_1-\eta_2$.
Thus $K^1(C(L^3_{q,p}))\cong \Z\oplus \Z/p$, generated by $\mathcal F_0$ (infinite order) and $\mathcal F_1-\mathcal F_0$ (order $p$).
\end{example}

\begin{example}
For $p=2$ (i.e., for the quantum real projective space $\R \mathrm{P}^{2n-1}_q$) one can likewise completely determine the $K$-homology. The coboundary is $\partial^\dual = D^2-1 = \sum_{i=1}^{n-1} (i+1)t^i$. Elementary computations reveal that the cokernel of $\partial^\dual$ is generated by $\eta_n$ (which has infinite order) and $(D-1)\eta_n$ (which has order $2^{n-1}$). Thus $K^1(C(\R\mathrm{P}^{2n-1}_q))\cong \Z\oplus \Z/2^{n-1}$, generated by $\mathcal F_0$ (infinite order) and $\mathcal F_1-\mathcal F_0$ (order $2^{n-1}$).
\end{example}

\begin{remark}
 The complex $\Z^n \xrightarrow{1-(1-t)^p} \Z^n$ also arises in \cite{ABL} from geometric considerations (a Gysin sequence coming from viewing quantum lens spaces as total spaces of principal $U(1)$-bundles over quantum complex projective spaces). This shows, incidentally, that the $K$-invariants of the quantum lens spaces are isomorphic to those of the corresponding classical spaces.
\end{remark}

\begin{remark}
 To conclude, let us sketch another computation of the $K$-theory and $K$-homology of quantum lens spaces, also using graph-algebra techniques. The idea is to view the quantum sphere, of which the lens space is a quotient, as the boundary of a quantum ball. The same line of argument applies in the commutative case: see  \cite[Corollary 2.7.6]{Atiyah}. 
 
 Fix $n$, $q$ and $p$ as above. Let ${G^+}$ be the graph obtained from $G=G_n$ by adding one vertex $v_{n+1}$, and for each $i=1,\ldots, n$ an edge $e_{i,n+1}$ with source $v_i$ and range $v_{n+1}$. The vertex $v_{n+1}\in C^*({G^+})$ generates an ideal $\Compact$ isomorphic to the compact operators, and $C^*({G^+})/\Compact \cong C^*(G)$. The algebra $C^*({G^+})$ can be interpreted in a natural way as the algebra of functions on a quantum ball $B^{2n}_q$, with the quotient map $C^*({G^+})\to C^*(G)$ corresponding to restriction of functions from $B^{2n}_q$ to its boundary (quantum) sphere $S^{2n-1}_q$: see \cite{HSballs}. 
 
 The automorphism $\alpha_p$ of $C^*(G)$ lifts to an automorphism of $C^*({G^+})$, which fixes the new vertex $v_{n+1}$. 
 The inclusion $\C v_{n+1} \into \Compact$ induces an isomorphism in equivariant $K$-theory
 \[K_*^{\Z/p}(\Compact)\cong K_*^{\Z/p}(\C)\cong \begin{cases} \Rep(\Z/p)\cong \Z[x]/(x^p-1)& *=0 \\ 0&*=1.\end{cases}\] ($\Rep$ denotes the representation ring.)
 The crossed product $C^*({G^+})\rtimes \Z/p$ can be identified with the $C^*$-algebra of a graph ${G^+}\rtimes \Z/p$ \cite{KP}, and a computation using Theorem \ref{Ktheory_theorem} shows that $K^{\Z/p}_*(C^*({G^+}))\cong K^{\Z/p}_*(\C)$, and that the endomorphism of $\Rep(\Z/p)$ induced by the inclusion $\Compact\into C^*({G^+})$ is multiplication by $\chi_n=(1-x)^n$ (the Euler characteristic of the representation of $\Z/p$ on $\C^n$ via multiplication by $\zeta_p$).
 The crossed product $C^*(G)\rtimes \Z/p$ is Morita equivalent to the fixed-point algebra $C^*(G)^{\alpha_p}\cong C(L)$, so the long exact sequence in equivariant $K$-theory induced by the restriction map $C^*(G^+)\to C^*(G)$ produces an identification of $K_*(C(L))$ with the homology of 
 \[ \Rep(\Z/p)\xrightarrow{ \chi_n } \Rep(\Z/p).\]
 A similar argument shows that the $K$-homology $K^*(C(L))$ is isomorphic to the cohomology of the dual complex.
\end{remark}

 \bibliography{k-hom}{}
 \bibliographystyle{alpha}

\end{document}